\documentclass[10pt,twoside]{article}
\usepackage{mathrsfs}
\usepackage{amsmath}
\usepackage{amssymb}
\usepackage{fancyhdr}
\usepackage{latexsym}
\usepackage{bbding}
\usepackage{mathrsfs}
\usepackage{wasysym}
\usepackage{cite}

\usepackage{multicol,graphics}

\newtheorem{theorem}{Theorem}[section]
\newtheorem{lemma}[theorem]{Lemma}
\newtheorem{corollary}[theorem]{Corollary}

\newtheorem{remark}[theorem]{Remark}

\numberwithin{equation}{section}
\newenvironment{proof}[1][Proof]{\noindent\textbf{#1.} }{\hfill $\Box$}
\allowdisplaybreaks

 \makeatletter
\begin{document}
\title{{On the temporal decay of solutions to the two-dimensional  nematic liquid crystal flows}
\thanks{This work is partially supported by the National Natural Science Foundation of China (11326155,
11401202), and by the Scientific Research Fund of Hunan Provincial
Education Department (14B117).}
}
\author{
{\small  
 Qiao Liu$^{1,2}$ \thanks{\text{E-mail address}: liuqao2005@163.com.
}}
\medskip
\\
{\small  $^{1}$\textit{Institute of Applied Physics and Computational Mathematics, Beijing, 100088, P. R. China}}\\
{\small  $^{2}$\textit{Department of Mathematical, Hunan Normal University, Changsha, Hunan 410081, P. R. China}}\\
}
\date{}
\maketitle

\begin{abstract}
We consider the temporal decay estimates for weak solutions to the two-dimensional nematic liquid crystal flows, and we show that
the energy norm of a global weak solution has non-uniform decay
\begin{align*}
\|u(t)\|_{L^{2}}+\|\nabla d(t)\|_{L^{2}}\rightarrow 0\quad \text{ as  } t\rightarrow \infty,
\end{align*}
under suitable conditions on the initial data. We also show the exact rate of the decay (uniform decay) of the energy norm  of the global weak solution.

\medskip


\end{abstract}

\section{Introduction}\label{Int}

\noindent

In this paper, we are interested in the large-time behavior of solutions to the following hydrodynamic system modeling the flow of  nematic liquid crystal materials in
two dimensions (see, e.g. \cite{L,LLW}):
\begin{align}
   \label{eq1.1}
&{\partial_{t}}u-\nu\Delta u +(u\cdot\nabla)u+\nabla{P}=-\lambda\nabla\cdot(\nabla d \odot\nabla d),\\
   \label{eq1.2}
&\partial_{t}d+(u\cdot\nabla)d=\gamma(\Delta d+|\nabla d|^{2}d),\\
   \label{eq1.3}
&\quad\quad\nabla\cdot u=0,\quad\quad |d|=1,\\
   \label{eq1.4}
&(u,d)|_{t=0}=(u_{0},d_{0}), \quad\quad |d_{0}|=1,
\end{align}
where $u(x,t):\mathbb{R}^{2}\times (0,+\infty)\rightarrow
\mathbb{R}^{2}$ is the unknown velocity field of the flow,
$P(x,t):\mathbb{R}^{2}\times (0,+\infty)\rightarrow \mathbb{R}$ is
the scalar pressure and $d:\mathbb{R}^{2}\times
(0,+\infty)\rightarrow \mathbb{S}^{1}$, the unit sphere in
$\mathbb{R}^{2}$, is the unknown (averaged) macroscopic/continuum
molecule orientation of the nematic liquid crystal flow, $\nabla
\cdot u=0$ represents the incompressible condition, $(u_{0},d_{0})$
is a given initial data with $\nabla \cdot u_{0}=0$ in distribution
sense, and  $\nu$, $\lambda$ and $\gamma$ are positive numbers
associated to the properties of the material: $\nu$ is the kinematic
viscosity, $\lambda$ is the competition between kinetic energy and
potential energy, and $\gamma$ is the microscopic elastic relaxation
time for the molecular orientation field. The notation $\nabla
d\odot\nabla d$ denotes the $2\times 2$ matrix whose $(i,j)$-th
entry is given by $\partial_{i}d\cdot
\partial_{j}d$ ($1\leq i,j\leq 2$). Since the concrete values of $\nu$, $\lambda$ and $\gamma$ do
not play a special role in our discussion, for simplicity, we assume
that they all equal to one throughout this paper.

The nematic liquid crystal flow \eqref{eq1.1}--\eqref{eq1.4} is a
simplified version of the Ericksen-Leslie model \cite{ER,LE}, but it
still retains most of the interesting mathematical properties.
Mathematically, system \eqref{eq1.1}--\eqref{eq1.4} is a strongly
coupled system between the incompressible Navier-Stokes (NS) equations
(the case $d\equiv \overline{d} _{0}$ ($\overline{d} _{0}$ is a constant vector in $\mathbb{S}^{1}$), e.g., \cite{PG,Le}) and the
transported heat flows of harmonic map (the case $u\equiv 0$, see
e.g., \cite{CDY,CS,MS,JL,W}), and thus, its mathematical analysis is
full of challenges.

 To make a clearer introduction to the results of the present paper, let us first concentrate on  the  NS equations.  Whether or not weak solutions decay to zero in $L^{2}$ as time
tends to infinity was posed by Leray in his pioneering paper \cite{Le}. Algebraic decay rates for the asymptotic behavior of solutions to the NS equations were first obtained by Schonbek \cite{S1}, using the method of Fourier splitting method, the author proved that there exists a weak solution to the n-dimensional ($n\geq3$)  NS equations in with initial data in $L^{1}\cap L^{2}$, satisfying
\begin{align*}
\|u(t)\|_{L^{2}}\leq C(1+t)^{-\frac{n}{2}+1},
\end{align*}
where the decay constant $C$ depends only on the $L^{1}$ and $L^{2}$ norms of the initial data. The Fourier splitting method was then extended by Schonbek \cite{S2} (see also Kajikiya and Miyakawa \cite{KM}) and it was proved that the decay rate of weak solutions to the three dimensional NS equations with initial data in $L^{p}\cap L^{2}$ with $1\leq p<2$ is same as that of the heat equation, i.e.,
\begin{align*}
\|u(t)\|_{L^{2}}\leq C(1+t)^{-\frac{3}{4}\left(\frac{2}{p}-1\right)},
\end{align*}
where the decay constant $C$ depends only on the $L^{p}$ and $L^{2}$ norms of the initial data.  Later, Zhang \cite{Z} established  that similar results of \cite{S1} still holds when the space dimension is two, and obtained that
\begin{align*}
\|u(t)\|_{L^{2}}\leq C(1+t)^{-\frac{1}{2}}.
\end{align*}
For more results on the decay of solutions to  the NS equations, we refer the
readers to \cite{S3,SW,ORS} and the references therein.

During the past several decades, there have been
many researches on system \eqref{eq1.1}--\eqref{eq1.4}, see, for
example,
\cite{JHW,HMC,HW1,LLZ,JT,L,LLW,LL1,LL2,LW,LW1,JL,LD,W,WD} and
the references therein. In what follows, we only briefly recall
some related results. In papers of Lin, Lin and Wang \cite{LLW} and
Hong \cite{HMC}, the authors proved that there exists global
Leray-Hopf type weak solutions to \eqref{eq1.1}--\eqref{eq1.4} with
suitable boundary condition in dimension two, and established that
the solutions are smooth away from at most finitely many singular
times which is similar as that for the heat flows of harmonic maps
(cf. \cite{CDY,CS,MS}). When the space dimension is three, Lin and Wang \cite{LW1} established the existence of global weak solutions very recently when the initial orientation $d_{0}$ maps to the up-hemisphere $\mathbb{S}^{2}_{+}$.
In \cite{W}, Wang proved the global existence of strong solution for rough initial data with sufficiently small $BMO^{-1}\times BMO$ norm, moreover,
 the global strong solution enjoys the decay estimate
 \begin{align*}
 \|u(t)\|_{L^{\infty}}+\|\nabla d(t)\|_{L^{\infty}} \leq C t^{-\frac{1}{2}}, \text{  for  } t>0.
 \end{align*}
 If the
qualification $|d|=1$ is deleted, more precisely, if the Dirichlet
energy
\begin{align*}
\frac{1}{2}\int_{\mathbb{R}^{n}}|\nabla d|^{2}\text{d}x\quad\text{
for  } d:\mathbb{R}^{n}\rightarrow \mathbb{S}^{2}
\end{align*}
 is replaced by the Ginzburg-Landau energy
\begin{align*}
\int_{\mathbb{R}^{n}}\left(\frac{1}{2}|\nabla
d|^{2}+\frac{(1-|d|^{2})^{2}}{4\varepsilon^{2}}\right)\text{d}x
\quad \text{ for } d:\mathbb{R}^{n}\rightarrow \mathbb{R}^{n},
\varepsilon>0.
\end{align*}
Lin and Liu \cite{LL1} proved the local classical solutions and the
global existence of weak solutions to this approximate system in dimensions two and three, and
for any fixed $\varepsilon$, they also obtained the existence and
uniqueness of global classical solution either in dimension two or
dimension three for large fluid viscosity $\nu$.  Moreover, a preliminary analysis of the asymptotic
behavior of global classical solution was also given in \cite{LL1}. Later in \cite{LL2}, Lin and Liu
proved partial regularity of the so-called suitable weak solutions in dimension three. In \cite{Wu}, by using the ${\L}$ojasiewicz-Simon approach method, Wu showed the uniqueness of asymptotic
limit of global classical solutions to the approximate liquid crystal flow, and  provided an estimate on the  convergence rate. More precisely,
it is show that under suitable initial conditions,  there exists unique classical solution to the approximate liquid crystal flow has the property
\begin{align*}
\lim_{t\rightarrow\infty}(\|u(t)\|_{H^{1}}+\|d(t)-d_{\infty}\|_{H^{2}})=0
\end{align*}
where $d_{\infty}$ is the steady state and satisfies  $-\Delta d_{\infty}+\frac{1}{\varepsilon^{2}}(|d_{\infty}|^{2}-1)d_{\infty}=0$.
Moreover, there exists a constant $\theta\in(0,\frac{1}{2})$ such that
\begin{align*}
\|u(t)\|_{H^{1}}+\|d(t)-d_{\infty}\|_{H^{2}}\leq C(1+t)^{-\frac{\theta}{1-2\theta}},
\end{align*}
where the  constant $C$ depends only on the initial data and $d_{\infty}$.

In this paper, motivated by the works cited above on the NS equations and nematic liquid crystal flows,  we study large-time behavior of solutions to the Cauchy problem of the two-dimensional nematic liquid crystal flow and establish temporal decay estimates for the solutions.
Our results, provide a mathematically rigorous basis to explain the decay of energy in the nematic liquid crystal flow, which  can be listed as follows.
\medskip

\begin{theorem}\label{thm1.1}
Let  $u_{0}\in  L^{2}(\mathbb{R}^{2})$ with
$\nabla\cdot u_{0}=0$, and $d_{0}\in
H^{1}(\mathbb{R}^{2};\mathbb{S}^{1})$ with $|d_{0}|=1$
and $d_{02}\geq \varepsilon_{0}>0$. Assume that
 $u\in L^{\infty}((0,\infty);L^{2}(\mathbb{R}^{2}))\cap L^{2}((0,\infty);H^{1}(\mathbb{R}^{2}))$
and $d\in  L^{\infty}((0,\infty);H^{1}(\mathbb{R}^{2},\mathbb{S}^{1}))$ is a global
weak solution to system \eqref{eq1.1}--\eqref{eq1.4}, then we have
\begin{align*}
\lim_{t\rightarrow \infty}\left(\|u(t)\|_{L^{2}}+\|\nabla d(t)\|_{L^{2}}\right)
=0.
\end{align*}
Here, $d_{02}$ denotes the second component of the initial orientation $d_{0}$.
\end{theorem}


\begin{theorem}\label{thm1.2}
\ Let $1\leq p<2$, for any initial data
$u_{0}\in L^{p}(\mathbb{R}^{2})\cap L^{2}(\mathbb{R}^{2})$ with
$\nabla\cdot u_{0}=0$, and $d_{0}\in
H^{1}(\mathbb{R}^{2};\mathbb{S}^{1})$ with $\nabla d_{0}\in L^{p}(\mathbb{R}^{2})$, $|d_{0}|=1$
and $d_{02}\geq \varepsilon_{0}>0$. Assume that  $u\in L^{\infty}((0,\infty);L^{2}(\mathbb{R}^{2}))\cap L^{2}((0,\infty);H^{1}(\mathbb{R}^{2}))$
and $d\in  L^{\infty}((0,\infty);H^{1}(\mathbb{R}^{2},\mathbb{S}^{1}))$ is a global weak solution to system \eqref{eq1.1}--\eqref{eq1.4}, then we have
\begin{align*}
\|u(t)\|_{L^{2}}+\|\nabla d(t)\|_{L^{2}}\leq C(1+t)^{-\frac{1}{2}\left(\frac{2}{p}-1\right)},
\end{align*}
where the constant $C$ depends on $\varepsilon_{0}$, and the $L^{p}$ and $L^{2}$ norms of the initial data.
\end{theorem}

\begin{remark}\label{rem1.3}
1.\ The main idea used in the proof of Theorem \ref{thm1.1} is similar as the paper of  Ogawa, Rajopadhye and Schonbek \cite{ORS} on prove decay in the context of the NS equations with slowly varying external forces. In order to do it, we first rewrite the energy norm in the frequency space and  divide it into two parts, then the decay of the low frequency part of energy norm is obtained through generalized energy inequality,  while the decay of the  high frequency part is established by using the Fourier splitting method.  To prove Theorem \ref{thm1.2}, we adapt the Fourier splitting method used in Schonbek \cite{S1} and  Zhang \cite{Z} on prove decay rate of the $L^{2}$ norm of solutions to the two dimensional NS equations when the initial data in $L^{1}\cap L^{2}$. Some of the new difficulties appears due to the fact that system \eqref{eq1.1}--\eqref{eq1.4} is the coupling of the NS equations and  the  transported heat flows of harmonic map. Here we  first get the bounds of $\widehat{u}$ and $\widehat{\nabla d}$ by taking  the Fourier transform on system, and then  prove a preliminary estimate on the decay of the energy norm of $(u,\nabla d)$ (see Lemma \ref{lem3.3} below), which we then use to establish the result of Theorem \ref{thm1.2}.
\medskip

2.\ When the space dimension becomes three, due to the strong nonlinearity  of the term $|\nabla d|^{2} d$ in equation \eqref{eq1.2}, it seems difficult to prove that similar results of Theorems \ref{thm1.1} and \ref{thm1.2} still hold for system \eqref{eq1.1}--\eqref{eq1.4}. We hope we can overcome this problem in the near future.
\end{remark}

The remaining parts of the present paper are organized as follows.
In Section 2,  we shall give the proof of Theorem \ref{thm1.1}, to
this end, some useful estimates are established. Section 3 is
devoted to proving Theorem \ref{thm1.2}. Throughout the paper, we use $\|\cdot\|_{X}$
to denote the norm of the scalar $X$-functions or the norm of the
$n$-vector $X$-functions,  and $\widehat{f}$ to denote the Fourier transform of function $f\in L^{2}(\mathbb{R}^{2})\cap L^{1} (\mathbb{R}^{2})$ defined by
\begin{align*}
\widehat{f}(\xi)= \frac{1}{2\pi}\int_{\mathbb{R}^{2}} f(x)e^{-ix\cdot\xi}\text{d}x,  \quad  i=\sqrt{-1}.
\end{align*}
We also denote by $C$ the positive
constant, which may depend on the initial data, and its value may change from line to line.

\section{The proof of Theorem \ref{thm1.1}}

In this section, we shall give the proof of Theorem \ref{thm1.1},  before going to do it, let us  recall the following rigidity theorem, which was recently established in Lei, Li and Zhang \cite{LLZ}, reads as follows.

\begin{theorem}\label{thm2.1}
Let $\varepsilon_{0}>0$, $C_{0}>0$. There exists a positive constant $\overline{\omega}=\overline{\omega}(\varepsilon_{0},C_{0})\in (0,1)$ such that the following holds:

 If $d:\mathbb{R}^{2}\rightarrow \mathbb{S}^{1}$,  $\nabla d \in H^{1}(\mathbb{R}^{2})$ with
$\|\nabla d\|_{L^{2}}\leq C_{0}$ and $d_{2}\geq \varepsilon_{0}$, then
\begin{align*}
\|\nabla d\|_{L^{4}}^{4}\leq (1-\overline{\omega})\|\Delta d\|_{L^{2}}^{2}.
\end{align*}

Consequently for such maps the associated harmonic energy is coercive, i.e.
\begin{align*}
\|\Delta d + |\nabla d|^{2} d\|_{L^{2}}^{2}\geq \frac{\overline{\omega}}{2} \left(\|\Delta d\|_{L^{2}}^{2}+\|\nabla d\|_{L^{4}}^{4}\right).
\end{align*}
\end{theorem}

In what follows, we shall establish some preliminary estimates which will be needed in the proof of Theorem \ref{thm1.1}.

\begin{lemma}\label{lem2.2}
Let $\psi \in C^{1}((0,\infty);C^{1}(\mathbb{R}^{2})\cap L^{2}(\mathbb{R}^{2}))$. Let $(u,d)$ be the global weak solution to system \eqref{eq1.1}--\eqref{eq1.4} defined as Theorem \ref{thm1.1}. Then for  $0<s<t$
\begin{align}\label{eq2.1}
\|\widehat{u}\widehat{\psi}(t)\|_{L^{2}}^{2}&\leq \|\widehat{u}\widehat{\psi}(s)\|_{L^{2}}^{2}
+2\int_{s}^{t}\left|\langle\widehat{\psi}' \widehat{u}(\tau),\widehat{\psi} \widehat{u}(\tau)\rangle-\||\xi|\widehat{u}\widehat{\psi}(\tau)\|_{L^{2}}^{2}\right|\text{d}\tau\nonumber\\
+&2\int_{s}^{t}\left|\langle\xi\cdot\widehat{u\otimes u}(\tau),\widehat{\psi}^{2} \widehat{u}(\tau)\rangle\right|\text{d}\tau
+2\int_{s}^{t}\left|\langle\xi\cdot\widehat{\nabla d\odot \nabla d }(\tau),\widehat{\psi}^{2} \widehat{u}(\tau)\rangle\right|\text{d}\tau;  \\
       \label{eq2.2}
\|\widehat{\nabla d}\widehat{\psi}(t)\|_{L^{2}}^{2}&\leq \|\widehat{\nabla d}\widehat{\psi}(s)\|_{L^{2}}^{2}
+2\int_{s}^{t}\left|\langle\widehat{\psi}' \widehat{\nabla d}(\tau),\widehat{\psi} \widehat{\nabla d}(\tau)\rangle-\||\xi|\widehat{\nabla d}\widehat{\psi}(\tau)\|_{L^{2}}^{2}\right|\text{d}\tau\nonumber\\
+&2\int_{s}^{t}\left|\langle\xi\cdot\widehat{u\cdot\nabla d}(\tau),\widehat{\psi}^{2} \widehat{\nabla d}(\tau)\rangle\right|\text{d}\tau
+2\int_{s}^{t}\left|\langle\xi\cdot\widehat{|\nabla d|^{2} d }(\tau),\widehat{\psi}^{2} \widehat{\nabla d}(\tau)\rangle\right|\text{d}\tau.
\end{align}
\end{lemma}

\begin{proof}
We first let $(u,d)$ be a smooth solution to system \eqref{eq1.1}--\eqref{eq1.4}. Taking the Fourier transform on \eqref{eq1.1}, multiplying the resulting equality  by $\widehat{\psi}^{2}\widehat{u}$ and integrating by parts, it follows that
\begin{align*}
\frac{d}{dt}\|\widehat{\psi}\widehat{u}(t)\|_{\!L^{\!2}}^{2}\!
=\!2\!\left(\!\langle\widehat{\psi}'\widehat{u}(t),\widehat{\psi}\widehat{u}(t)\rangle\!-\!\||\xi|\widehat{\psi}\widehat{u}(t)\|_{\!L^{\!2}}^{2}\!\right)
\!-\!2\!\langle\xi\cdot\widehat{u\otimes u}(t),\widehat{\psi}^{2}\widehat{u}(t)\rangle\!-\!2\!\langle\xi \cdot\widehat{\nabla d\odot \nabla d}(t),\widehat{\psi}^{2}\widehat{u}(t)\rangle,
\end{align*}
where we have used the fact that the divergence free condition \eqref{eq1.3} implies that $u\cdot \nabla u=\nabla(u\otimes u)$ and $\langle\widehat{\nabla P},\widehat{\psi}^{2} \widehat{u}\rangle=0$. Integrating the above equality  with respect to times variable  between $s$ and $t$ yields
\begin{align*}
\|\widehat{u}\widehat{\psi}(t)\|_{L^{2}}^{2}&\leq \|\widehat{u}\widehat{\psi}(s)\|_{L^{2}}^{2}
+2\int_{s}^{t}\left|\langle\widehat{\psi}' \widehat{u}(\tau),\widehat{\psi} \widehat{u}(\tau)\rangle-\||\xi|\widehat{u}\widehat{\psi}(\tau)\|_{L^{2}}^{2}\right|\text{d}\tau\nonumber\\
+&2\int_{s}^{t}\left|\langle\xi\cdot\widehat{u\otimes u}(\tau),\widehat{\psi}^{2} \widehat{u}(\tau)\rangle\right|\text{d}\tau
+2\int_{s}^{t}\left|\langle\xi\cdot\widehat{\nabla d\odot \nabla d }(\tau),\widehat{\psi}^{2} \widehat{u}(\tau)\rangle\right|\text{d}\tau.
\end{align*}
To prove \eqref{eq2.2}, by applying $\nabla$ to \eqref{eq1.2}, we have
\begin{align}\label{eq2.3}
\nabla d_{t}-\Delta \nabla d+\nabla (u\cdot\nabla d)=\nabla(|\nabla d|^{2}d).
\end{align}
Similarly, by taking the Fourier transform on \eqref{eq2.3}, multiplying the resulting equality by $\widehat{\psi}^{2}\widehat{\nabla d}$, and integrating with
respect to space variable $x$ on $\mathbb{R}^{2}$ and then with respect to time variable $t$ between $s$ and $t$, we get \eqref{eq2.2}.

 By using the standard mollifiers method, we can extend the estimates  \eqref{eq2.1} and
  \eqref{eq2.2} to weak solutions. For the details, we refer the readers to see Ogawa, Rajopadhye and Schonbek \cite{ORS}. Thus we complete the proof of  Lemma \ref{lem2.2}.
\end{proof}

\begin{corollary}\label{cor2.3}
Let $\phi\in C^{1}((0,\infty);L^{2}(\mathbb{R}^{2}))$. Let $(u,d)$ be the global weak solution to system \eqref{eq1.1}--\eqref{eq1.4} defined as Theorem \ref{thm1.1}.  Then for $0<s<t$
\begin{align}\label{eq2.4}
\|\widehat{u}\phi(t)\|_{L^{2}}^{2}\leq& \|\widehat{u}(s)e^{-|\xi|^{2}(t-s)}\phi(t)\|_{L^{2}}^{2}+2\int_{s}^{t} \left|\langle\xi\cdot\widehat{(u\otimes u)}(\tau),e^{-2|\xi|^{2}(t-\tau)}\phi^{2}(t)\widehat{u}(\tau)\rangle\right|\text{d}\tau\nonumber\\
&+2\int_{s}^{t} \left|\langle\xi\cdot\widehat{(\nabla d\otimes \nabla d)}(\tau),e^{-2|\xi|^{2}(t-\tau)}\phi^{2}(t)\widehat{u}(\tau)\rangle\right|\text{d}\tau;\\
          \label{eq2.5}
\|\widehat{\nabla d}\phi(t)\|_{L^{2}}^{2}\leq& \|\widehat{\nabla d}(s)e^{-|\xi|^{2}(t-s)}\phi(t)\|_{L^{2}}^{2}+2\int_{s}^{t} \left|\langle\xi\cdot\widehat{(u\cdot\nabla d)}(\tau),e^{-2|\xi|^{2}(t-\tau)}\phi^{2}(t)\widehat{\nabla d}(\tau)\rangle\right|\text{d}\tau\nonumber\\
&+2\int_{s}^{t} \left|\langle\xi\cdot\widehat{(|\nabla d|^{2} d)}(\tau),e^{-2|\xi|^{2}(t-\tau)}\phi^{2}(t)\widehat{\nabla d}(\tau)\rangle\right|\text{d}\tau.
\end{align}
\end{corollary}

\begin{proof}
Take $\widehat{\psi}_{\eta}(\tau):= e^{-|\xi|^{2}(t+\eta-\tau)}\phi(\xi,t)$ for $\eta>0$. Then we have
\begin{align*}
\langle\widehat{\psi}'_{\eta} \widehat{f}(\tau),\widehat{\psi}_{\eta} \widehat{f}(\tau)\rangle
=\langle|\xi|^{2}\widehat{\psi}_{\eta} \widehat{f}(\tau),\widehat{\psi}_{\eta} \widehat{f}(\tau)\rangle
=\||\xi|\widehat{\psi}'_{\eta} \widehat{f}(\tau)\|_{L^{2}}^{2}\quad \text{ for all } f\in L^{2}(\mathbb{R}^{2}).
\end{align*}
Hence, the first integrand in the right hand side of \eqref{eq2.1} and \eqref{eq2.2} vanishes. Taking limit as $\eta\rightarrow 0$, we see that
\begin{align*}
\widehat{\psi}(t) ={\phi}(\xi,t)\quad \text{  and  } \widehat{\psi}(s)=e^{-|\xi|^{2}(t-s)}\phi(\xi,t),
\end{align*}
which together with \eqref{eq2.1} and \eqref{eq2.2} ensure \eqref{eq2.4} and \eqref{eq2.5}. Thus we complete the proof of Corollary \ref{cor2.3}.
\end{proof}

\begin{lemma}\label{lem2.4}
Let $E\in C^{1}((0,\infty);\mathbb{R})$ and $\psi\in C^{1}((0,\infty);L^{\infty}(\mathbb{R}^{2}))$ such that $\psi^{2}\in L^{\infty}((0,\infty);L^{\infty}(\mathbb{R}^{2}))$ and $\nabla \mathcal{F}^{-1}(\psi^{2})\in L^{\infty}((0,\infty);L^{2}(\mathbb{R}^{2}))$.
 Let $(u,d)$ be the global weak solution to system \eqref{eq1.1}--\eqref{eq1.4} defined as Theorem \ref{thm1.1}. Then
 \begin{align}\label{eq2.6}
 E(t)\|\psi\widehat{u}(t)\|_{L^{2}}^{2} \leq &E(s)\|\psi \widehat{u}(s)\|_{L^{2}}^{2}+\!\int_{s}^{t} \! E'(\tau)\|\psi\widehat{u}(\tau)\|_{L^{2}}^{2}\text{d}\tau+2\!\int_{s}^{t}\! E(\tau) \left(\langle\psi'\widehat{u}(\tau),\psi\widehat{u}(\tau)\rangle
 -\||\xi|\psi\widehat{u}(\tau)\|_{L^{2}}^{2}\right)\text{d}\tau\nonumber\\
 \!\!+ 2\int_{s}^{t}& E(\tau)\left|\langle\xi\cdot\widehat{u\otimes u}(\tau),\psi^{2}(\tau)\widehat{u}(\tau)\rangle\right|\text{d}\tau
 +2\int_{s}^{t}E(\tau)\left|\langle\xi\cdot\widehat{\nabla d\odot\nabla d}(\tau),\psi^{2}(\tau)\widehat{u}(\tau)\rangle\right|\text{d}\tau;\\
 \medskip
            \label{eq2.7}
 E(t)\|\psi\widehat{\nabla d}(t)\|_{\!L^{\!2}}^{2}\!\leq & E(s)\|\psi \widehat{\nabla d}(s)\|_{\!L^{\!2}}^{2}\!+\!\!\!\int_{s}^{t} \!\! \! E'(\tau)\|\psi\widehat{\nabla d}(\tau)\|_{\!L^{\!2}}^{2}\text{d}\tau\!+\!2\!\!\int_{s}^{t}\!\!\! E(\tau)\! \left(\!\!\langle\psi'\widehat{\nabla d}(\tau),\psi\widehat{\nabla d}(\tau)\rangle
 \!-\!\||\xi|\psi\widehat{\nabla d}(\tau)\|_{\!L^{\!2}}^{2}\!\right)\!\text{d}\tau\nonumber\\
 + 2\int_{s}^{t}& E(\tau)\left|\langle\xi\cdot\widehat{u\cdot\nabla d}(\tau),\psi^{2}(\tau)\widehat{u}(\tau)\rangle\right|\text{d}\tau
 +2\int_{s}^{t}E(\tau)\left|\langle\xi\cdot\widehat{|\nabla d|^{2} d}(\tau),\psi^{2}(\tau)\widehat{u}(\tau)\rangle\right|\text{d}\tau.
 \end{align}
\end{lemma}

 \begin{proof}
We first prove the estimates \eqref{eq2.6} and \eqref{eq2.7} for smooth solutions. As in Lemma \ref{lem2.2}, we take the Fourier transform of \eqref{eq1.1}, multiply it by
$E\psi^{2} \widehat{u}$, integrate it with respect to $x$ and then with respect to $t$ between $s$ and $t$, we obtain the formal estimate \eqref{eq2.6}. In a similar way, by taking the Fourier transform of \eqref{eq2.3}, multiplying it by
$E\psi^{2} \widehat{\nabla d}$,  integrating it with respect to $x$ and then with respect to $t$  between $s$ and $t$, we obtain \eqref{eq2.7}. When using the retarded mollifiers method, the conditions $\psi^{2}\in L^{\infty}((0,\infty);L^{\infty}(\mathbb{R}^{2}))$ and $\nabla \mathcal{F}^{-1}(\psi^{2})\in L^{\infty}((0,\infty);L^{2}(\mathbb{R}^{2}))$ will guarantee the weak convergence of the nonlinear term. For the details we see  Ogawa, Rajopadhye and Schonbek \cite{ORS}.
\end{proof}
\medskip
\\
\textbf{Proof of Theorem \ref{thm1.1}.}
In what follows, we adapt the argument used in Ogawa, Rajopadhye and Schonbek \cite{ORS} to prove decay of weak solution to the Navier--Stokes equations with slowly varying external forces. We shall  split the proof into two steps, i.e., the estimates for the low frequency part of the energy and for the high frequency part. Let $(u, d)$ be a solution to \eqref{eq1.1}--\eqref{eq1.4}. For $\phi=\phi(\xi,t)$, we have
\begin{align*}
\|(\widehat{u}(t),\widehat{\nabla d}(t))\|_{L^{2}}^{2}\leq 2\left(\|(\phi(t)\widehat{u}(t),\phi(t)\widehat{\nabla d}(t))\|_{L^{2}}^{2}+\|((1-\phi(t))\widehat{u}(t),(1-\phi(t))\widehat{\nabla d}(t))\|_{L^{2}}^{2}\right).
\end{align*}
We call the terms $\|(\phi(t)\widehat{u}(t),\phi(t)\widehat{\nabla d}(t))\|_{L^{2}}$ and $\|((1-\phi(t))\widehat{u}(t),(1-\phi(t))\widehat{\nabla d}(t))\|_{L^{2}}$ the low and high frequency parts of the energy respectively.
\medskip

\textbf{\textit{Step 1.} low frequency part energy estimate.} Before going to do it, we first notice that, for a weak solution
to system \eqref{eq1.1}--\eqref{eq1.4}, one has the following basic energy law (see \cite{LLW})
\begin{align}\label{eq2.8}
\!\|u(t)\|_{\!L^{2}}^{2}\! +\!\|\nabla d(t)\|_{\!L^{\!2}}^{2}\!+\!2\!\!\int_{0}^{t}\!\!
\left(\|\nabla u(\tau)\|_{\!L^{\!2}}^{2}\!+\!\|(\Delta d\!+\!|\nabla d|^{2}d)(\tau)\|_{\!L^{\!2}}^{2}\!\right)\text{d}\tau\!=\!\|u_{0}\|_{\!L^{\!2}}^{2}\!+\!\|\nabla d_{0}\|_{\!L^{\!2}}^{2}, \text{ for all } 0\!<\!t\!\leq\infty.
\end{align}
We also notice that if the initial orientation $d_{02}\geq \varepsilon_{0}$ for some positive $\varepsilon_{0}$, by using the standard maximum principle to the second component of $d_{2}$,  we get
\begin{align}\label{eq2.9}
d_{2}\geq \varepsilon_{0},\quad \text{for all } t>0.
\end{align}
By using \eqref{eq2.9},  Theorem \ref{thm2.1}  and energy equality \eqref{eq2.8},  one has the following energy inequality
\begin{align}\label{eq2.10}
\!\|u(t)\|_{\!L^{2}}^{2}\! +\!\|\nabla d(t)\|_{\!L^{\!2}}^{2}\!+\!\overline{\omega}\!\int_{0}^{t}\!\!
\left(\|\nabla u(\tau)\|_{\!L^{\!2}}^{2}\!+\!\|\Delta d(\tau)\|_{\!L^{\!2}}^{2}\!\right)\text{d}\tau\leq\|u_{0}\|_{\!L^{\!2}}^{2}\!+\!\|\nabla d_{0}\|_{\!L^{\!2}}^{2}, \text{ for all } 0\!<\!t\!\leq\infty.
\end{align}
In what follows, we choose
\begin{align*}
\phi(\xi,t):=e^{-|\xi|^{2}t},
\end{align*}
by  applying Corollary \ref{cor2.3} with $\phi$ defined as above equality, we obtain
\begin{align}\label{eq2.11}
&\|\widehat{u}(t)\phi(t)\|_{L^{2}}^{2} +\|\widehat{\nabla d}(t)\phi(t)\|_{L^{2}}^{2}
\leq\|\widehat{u}(s)\phi(t-s)\phi(t)\|_{L^{2}}^{2}+\|\widehat{\nabla d}(s)\phi(t-s)\phi(t)\|_{L^{2}}^{2}\nonumber\\
+&2\int_{s}^{t} \left|\langle\xi\cdot\widehat{(u\otimes u)}(\tau),\phi^{2}(t-\tau)\phi^{2}(t)\widehat{u}(\tau)\rangle\right|\text{d}\tau
+2\int_{s}^{t} \left|\langle\xi\cdot\widehat{(\nabla d\otimes \nabla d)}(\tau),\phi^{2}(t-\tau)\phi^{2}(t)\widehat{u}(\tau)\rangle\right|\text{d}\tau\nonumber\\
+&2\int_{s}^{t} \left|\langle\xi\cdot\widehat{(u\cdot\nabla d)}(\tau),\phi^{2}(t-\tau)\phi^{2}(t)\widehat{\nabla d}(\tau)\rangle\right|\text{d}\tau+2\int_{s}^{t} \left|\langle\xi\cdot\widehat{(|\nabla d|^{2} d)}(\tau),\phi^{2}(t-\tau)\phi^{2}(t)\widehat{\nabla d}(\tau)\rangle\right|\text{d}\tau\nonumber\\
:= & \|\widehat{u}(s)\phi(t-s)\phi(t)\|_{L^{2}}^{2}+\|\widehat{\nabla d}(s)\phi(t-s)\phi(t)\|_{L^{2}}^{2}
+I_{1}(t)+I_{2}(t)+I_{3}(t)+I_{4}(t).
\end{align}
By using the Dominated Convergence Theorem, it is easy to see
\begin{align*}
\overline{\lim_{t\rightarrow \infty} }\|\widehat{u}(s)\phi(t-s)\phi(t)\|_{L^{2}}^{2}=0
\quad \text{and } \overline{\lim_{t\rightarrow \infty}} \|\widehat{\nabla d}(s)\phi(t-s)\phi(t)\|_{L^{2}}^{2}=0.
\end{align*}
By the H\"{o}lder inequality and the interpolation inequality, it follows that
\begin{align*}
\left|\langle\xi\cdot\widehat{(u\otimes u)}(\tau),\phi^{2}(t-\tau)\phi^{2}(t)\widehat{u}(\tau)\rangle\right|
\leq&  \left|\langle \widehat{(u\otimes u)}(\tau),\phi^{2}(t-\tau)\phi^{2}(t)|\xi|\widehat{u}(\tau)\rangle\right|\nonumber\\
\leq&\|\widehat{(u\otimes u)}\|_{L^{2}} \|\phi^{2}(t-\tau)\phi^{2}(t)\|_{L^{\infty}}\||\xi|\widehat{u}\|_{L^{2}}\nonumber\\
=&\|(u\otimes u)\|_{L^{2}} \|\phi^{2}(t-\tau)\phi^{2}(t)\|_{L^{\infty}}\|\nabla{u}\|_{L^{2}}\nonumber\\
\leq& C \|u\|_{L^{4}}^{2}\|\nabla{u}\|_{L^{2}}\leq C\|u\|_{L^{2}}\|\nabla{u}\|_{L^{2}}^{2}.
\end{align*}
Hence, we have
\begin{align*}
I_{1}(t)\leq C\int_{s}^{t} \|u(\tau)\|_{L^{2}}\|\nabla{u}(\tau)\|_{L^{2}}^{2}\text{d}\tau\leq C\int_{s}^{t} \|\nabla{u}(\tau)\|_{L^{2}}^{2}\text{d}\tau,
\end{align*}
where we have used the energy inequality \eqref{eq2.10}  in the last inequality. In a similar way as derive the estimations of $I_{1}$, it is easy to see that
\begin{align*}
I_{2}(t)+I_{3}(t)+I_{4}(t)\leq C\int_{s}^{t} (\|\nabla u(\tau)\|_{L^{2}}^{2}+\|\nabla^{2} d(\tau)\|_{L^{2}}^{2})\text{d}\tau.
\end{align*}
Therefore, by taking a limit $t\rightarrow\infty$ in \eqref{eq2.11} that
\begin{align*}
\overline{\lim_{t\rightarrow \infty} } \left(\|\widehat{u}(t)\phi(t)\|_{L^{2}}^{2} +\|\widehat{\nabla d}(t)\phi(t)\|_{L^{2}}^{2}\right)
\leq C\int_{s}^{\infty} (\|\nabla u(\tau)\|_{L^{2}}^{2}+\|\nabla^{2} d(\tau)\|_{L^{2}}^{2})\text{d}\tau.
\end{align*}
Since the right hand side the above inequality convergence to $0$ as $s\rightarrow \infty$, we obtain that the low frequency part of the energy goes to zero.
\medskip

\textbf{\textit{Step 2.} high  frequency part energy estimate.}  Let $\psi(\xi,t)=1-\phi(\xi,t)$. By applying Lemma \ref{lem2.4}, after rearranging terms, we obtain
\begin{align*}
&\|\psi(t)\widehat{u}(t)\|_{L^{2}}^{2}+\|\psi(t)\widehat{\nabla d}(t)\|_{L^{2}}^{2}\leq \frac{E(s)}{E(t)}\left(\|\psi \widehat{u}(s)\|_{L^{2}}^{2}+\|\psi \widehat{\nabla d}(s)\|_{\!L^{\!2}}^{2}\right)\nonumber\\
+&\frac{1}{E(t)} \!\int_{s}^{t} \!\left( E'(\tau)\|\psi\widehat{u}(\tau)\|_{L^{2}}^{2}-2 E(\tau)\||\xi|\psi\widehat{u}(\tau)\|_{L^{2}}^{2}\right)+\left( E'(\tau)\|\psi\widehat{\nabla d}(\tau)\|_{\!L^{2}}^{2}-2E(\tau) \||\xi|\psi\widehat{\nabla d}(\tau)\|_{\!L^{2}}^{2}\right)\text{d}\tau\nonumber\\
+&\frac{2}{E(t)} \!\int_{s}^{t} \! E(\tau)\left(\langle\psi'\widehat{u}(\tau),\psi\widehat{u}(\tau)\rangle+ \langle\psi'\widehat{\nabla d}(\tau),\psi\widehat{\nabla d}(\tau)\rangle
\right)\text{d}\tau \nonumber\\
+&\frac{2}{E(t)} \!\int_{s}^{t} E(\tau)\left|\langle\xi\cdot\widehat{u\otimes u}(\tau),\psi^{2}(\tau)\widehat{u}(\tau)\rangle\right|\text{d}\tau
 +\frac{2}{E(t)}\int_{s}^{t}E(\tau)\left|\langle\xi\cdot\widehat{\nabla d\odot\nabla d}(\tau),\psi^{2}(\tau)\widehat{u}(\tau)\rangle\right|\text{d}\tau;\nonumber\\
 +&\frac{2}{E(t)} \int_{s}^{t} E(\tau)\left|\langle\xi\cdot\widehat{u\cdot\nabla d}(\tau),\psi^{2}(\tau)\widehat{\nabla d}(\tau)\rangle\right|\text{d}\tau
 +\frac{2}{E(t)}\int_{s}^{t}E(\tau)\left|\langle\xi\cdot\widehat{|\nabla d|^{2} d}(\tau),\psi^{2}(\tau)\widehat{\nabla d}(\tau)\rangle\right|\text{d}\tau\nonumber\\
 :=&II_{1}(t)+II_{2}(t)+\cdots +II_{7}(t).
\end{align*}
By choose $E(t)=(1+t)^{k}$ with $k>2$, we will estimate $II_{i}(t) (i=1,2,\cdots,7)$ term by term. Notice that $|\psi|\leq 1$ and energy inequality \eqref{eq2.10}, we have
\begin{align*}
II_{1}(t)=& \left(\frac{1+s}{1+t}\right)^{k}\left(\|\psi(s) \widehat{u}(s)\|_{L^{2}}^{2}+\|\psi(s) \widehat{\nabla d}(s)\|_{L^{2}}^{2}\right)
\leq \left(\frac{1+s}{1+t}\right)^{k}\left(\|\widehat{u}(s)\|_{L^{2}}^{2}+\|\widehat{\nabla d}(s)\|_{L^{2}}^{2}\right)\nonumber\\
=&\left(\frac{1+s}{1+t}\right)^{k}\left(\|u(s)\|_{L^{2}}^{2}+\|\nabla d(s)\|_{L^{2}}^{2}\right)\leq C\left(\frac{1+s}{1+t}\right)^{k}.
\end{align*}
Thus we have
\begin{align*}
\lim_{t\rightarrow\infty} II_{1}(t)=0.
\end{align*}
For the term $II_{2}$, by using the Fourier splitting method, let
\begin{align*}
B(t):=\{\xi\in\mathbb{R}^{2}: |\xi|\leq G(t)\}
\end{align*}
where $G$ is to be determined below. Then
\begin{align*}
&\left( E'(\tau)\|\psi\widehat{u}(\tau)\|_{L^{2}}^{2}-2 E(\tau)\||\xi|\psi\widehat{u}(\tau)\|_{L^{2}}^{2}\right)+\left( E'(\tau)\|\psi\widehat{\nabla d}(\tau)\|_{\!L^{2}}^{2}-2E(\tau) \||\xi|\psi\widehat{\nabla d}(\tau)\|_{\!L^{2}}^{2}\right)\nonumber\\
=& E'(\tau)\int_{\mathbb{R}^{2}\backslash B(t)}\left(|\psi \widehat{u}(\tau)|^{2}+|\psi\widehat{\nabla d}(\tau)|^{2}\right)\text{d}\xi -2E(\tau) \int_{\mathbb{R}^{2}\backslash B(t)} |\xi|^{2}\left(|\psi\widehat{u}(\tau)|^{2}+|\psi\widehat{\nabla d}(\tau)|^{2}\right)\text{d}\xi\nonumber\\
&+E'(\tau)\int_{ B(t)}\left(|\psi \widehat{u}(\tau)|^{2}+|\psi\widehat{\nabla d}(\tau)|^{2}\right)\text{d}\xi -2E(\tau) \int_{B(t)} |\xi|^{2}\left(|\psi\widehat{u}(\tau)|^{2}+|\psi\widehat{\nabla d}(\tau)|^{2}\right)\text{d}\xi\nonumber\\
\leq & \left( E'(\tau)-\!2E(\tau) G^{2}(\tau)\right)\!\!\int_{\!\mathbb{R}^{2}\backslash B(t)}\!\!\left(|\psi \widehat{u}(\tau)|^{2}\!+|\psi\widehat{\nabla d}(\tau)|^{2}\right)\text{d}\xi+\!E'(\tau)\!\int_{ B(t)}\!\!\left(|\psi \widehat{u}(\tau)|^{2}\!+|\psi\widehat{\nabla d}(\tau)|^{2}\right)\text{d}\xi ,
\end{align*}
where we have used the fact that the integer $E(\tau)\int_{B(t)} |\xi|^{2}\left(|\psi\widehat{u}(\tau)|^{2}+|\psi\widehat{\nabla d}(\tau)|^{2}\right)\text{d}\xi$ is nonnegative in the last inequality. Choosing $G(t)=\left(\frac{k}{2(1+t)}\right)^{\frac{1}{2}}$, we see that $E'(\tau)-2E(\tau) G^{2}(\tau)=0$, which implies that the first term in the right hand side of the above inequality vanishes. Hence, we have
\begin{align*}
II_{2}(t)\leq& \frac{k}{(1+t)^{k}}\int_{s}^{t} (1+\tau)^{k-1}\int_{B(t)}\!\left(|\psi \widehat{u}(\tau)|^{2}\!+|\psi\widehat{\nabla d}(\tau)|^{2}\right)\text{d}\xi\text{d}\tau.
\end{align*}
Observing that $\psi(\xi,t)=1-e^{-|\xi|^{2}t}$,  then we have  $|\psi|\leq |\xi|^{2}$ for  $|\xi|\leq 1$. Then
\begin{align*}
&\int_{B(t)}\!\left(|\psi \widehat{u}(\tau)|^{2}\!+|\psi\widehat{\nabla d}(\tau)|^{2}\right)\text{d}\xi
\leq \int_{B(t)}|\xi|^{4}\left(|\widehat{u}(\tau)|^{2}\!+|\widehat{\nabla d}(\tau)|^{2}\right)\text{d}\xi\nonumber\\
\leq& G^{4}(t) \int_{\mathbb{R}^{2}}\left(|\widehat{u}(\tau)|^{2}\!+|\widehat{\nabla d}(\tau)|^{2}\right)\text{d}\xi
=G^{4}(t)(\|u(\tau)\|_{L^{2}}^{2}+\|\nabla d(\tau)\|_{L^{2}}^{2})\leq CG^{4}(t)=\frac{C}{(1+t)^{2}},
\end{align*}
where we have used the energy inequality \eqref{eq2.10} in the last inequality. Hence
\begin{align*}
II_{2}(t)\leq& \frac{C}{(1+t)^{k}}\int_{s}^{t} (1+\tau)^{k-3}\text{d}\tau\leq \frac{C}{(1+t)^{2}},
\end{align*}
which implies that
\begin{align*}
\lim_{t\rightarrow \infty}II_{2}(t)=0.
\end{align*}
For the term $II_{3}$. Notice that there hold $\psi'(\xi,t)=\frac{\partial\psi}{\partial t}(\xi,t)=|\xi|^{2}e^{-|\xi|^{2}t}=|\xi|^{2}\phi(\xi,t)$, and $E(t)$ is an increasing function on $t$, we have
\begin{align*}
II_{3}(t)=&\frac{2}{E(t)}\int_{s}^{t}\left(\langle|\xi|^{2}\phi(\tau)\widehat{u}(\tau), (1-\phi(\tau))\widehat{u}(\tau)\rangle+\langle|\xi|^{2}\phi(\tau)\widehat{\nabla d}(\tau), (1-\phi(\tau))\widehat{\nabla d}(\tau)\rangle\right) \text{d}\tau\nonumber\\
\leq& 2\int_{s}^{t} \left(\langle|\xi|\widehat{u}(\tau), |\xi|\widehat{u}(\tau)\rangle +\langle|\xi|\widehat{\nabla d}(\tau), |\xi|\widehat{\nabla d}(\tau)\rangle\right) \text{d}\tau\leq 2\int_{s}^{t} (\|\nabla u(\tau)\|_{L^{2}}^{2}+\|\nabla^{2} d\|_{L^{2}}^{2})\text{d}\tau.
\end{align*}
By letting $s$ and $t$ go to infinity, we obtain
\begin{align*}
\lim_{t\rightarrow\infty} II_{3}(t) =0.
\end{align*}
For the term $II_{4}$, notice that $|\psi|\leq 1$ and $E(t)$ is an increasing function on $t$,
\begin{align*}
II_{4} (t)\leq& \frac{2}{E(t)}\int_{s}^{t} E(\tau) \left|\langle \psi \widehat{u\otimes u}(\tau), \psi(\tau) |\xi|\widehat{u}(\tau)\rangle\right|\text{d}\tau\leq C\int_{s}^{t} \|\psi(\tau)\widehat{u\otimes u}(\tau)\|_{L^{2}}
\|\psi(\tau)|\xi|\widehat{u}(\tau)\|_{L^{2}}\text{d}\tau\nonumber\\
\leq &C\int_{s}^{t} \|\widehat{u\otimes u}(\tau)\|_{L^{2}}
\||\xi|\widehat{u}(\tau)\|_{L^{2}}\text{d}\tau\leq C\int_{s}^{t} \|u(\tau)\|_{L^{4}}^{2}
\|\nabla u(\tau)\|_{L^{2}}\text{d}\tau\nonumber\\
\leq& C\int_{s}^{t} \|u(\tau)\|_{L^{2}}
\|\nabla u(\tau)\|_{L^{2}}^{2}\text{d}\tau\leq C\int_{s}^{t} \|\nabla u(\tau)\|_{L^{2}}^{2}\text{d}\tau.
\end{align*}
As before, by letting $s$ and $t$ go to infinity, we get
\begin{align*}
\lim_{t\rightarrow\infty} II_{4} (t)=0.
\end{align*}
Similar as the derivation of the estimates of $II_{3}$, it is easy to see
\begin{align*}
II_{5}(t)+II_{6}(t)+II_{7}(t)\leq &\!C\!\!\int_{s}^{t}\! \left(\|u(\tau)\|_{L^{2}}+\|\nabla d(\tau)\|_{L^{2}}\right)
\left(\|\nabla u(\tau)\|_{L^{2}}^{2}\!+\|\nabla^{2}d(\tau)\|_{L^{2}}^{2}\right)\text{d}\tau\nonumber\\
\leq& \!C\!\int_{s}^{t}\!\left( \|\nabla u(\tau)\|_{L^{2}}^{2}+\|\nabla^{2}d(\tau)\|_{L^{2}}^{2}\right)\text{d}\tau.
\end{align*}
Therefore by letting $s$ and $t$ go to infinity, we get
\begin{align*}
\lim_{t\rightarrow\infty} (II_{5}(t)+II_{6}(t)+II_{7}(t)) =0.
\end{align*}
Thus the high frequency part of the energy norm goes to zero, which concludes the proof of Theorem \ref{thm1.1}.\hfill$\Box$

\section{The proof of Theorem \ref{thm1.2}}

In this Section, by using the argument  in Schonbek \cite{S1} and Zhang \cite{Z} on study the temporal decay rate of solutions to the  NS equations, we shall give the proof of Theorem \ref{thm1.2}. We first give two preliminary estimate, which well be necessary in the sequel.

\begin{lemma}\label{lem3.1}
Let $(u,d)$ be a solution to system \eqref{eq1.1}--\eqref{eq1.4} with initial data $(u_{0}, d_{0})$ satisfying the initial data
as in Theorem \ref{thm1.2}.  Then we have
\begin{align}\label{eq3.1}
|\widehat{u} (t)|+|\widehat{\nabla d}(t)|\leq  \left( |\widehat{u}_{0}|+|\widehat{\nabla d_{0}}|+ 2|\xi|\int_{0}^{t} (\|u (\tau)\|_{L^{2}}^{2}+\|\nabla d(\tau)\|_{L^{2}}^{2})\text{d}\tau\right).
\end{align}
\end{lemma}

\begin{proof}
By using the elementary vector calculus, one can rewrite \eqref{eq1.1}  as
\begin{align}\label{eq3.2}
u_{t}-\Delta u =-\mathbb{P}\nabla \cdot (u\otimes u-\nabla d\odot\nabla d),
\end{align}
where $\mathbb{P}$ is the Leray projection operator defined by $\mathbb{P}f=f-\nabla\Delta^{-1}\nabla\cdot f$.
Taking Fourier transform on both side of \eqref{eq3.2} and \eqref{eq2.3}, we have the following representation of solutions in terms of the Fourier transform
\begin{align*}
\widehat{u}(t)= e^{-t|\xi|^{2}} \widehat{u}_{0}- \!\!\int_{0}^{t} e^{-(t-\tau)|\xi|^{2}}\!
\left(1-\frac{\xi \otimes \xi}{|\xi|^{2}}\right)\xi\cdot \left\{\widehat{(u\otimes u)}(\tau)+\widehat{(\nabla d\odot \nabla d)}(\tau)\right\}\text{d}\tau,
\end{align*}
and
\begin{align*}
\widehat{\nabla d}(t)= e^{-t|\xi|^{2}} \widehat{\nabla d}_{0}- \!\int_{0}^{t} e^{-(t-\tau)|\xi|^{2}}
\xi\cdot \left\{\widehat{(u\cdot \nabla d)}(\tau)+\widehat{(|\nabla d|^{2} d)}(\tau)\right\}\text{d}\tau,
\end{align*}
respectively.
From the above two representation,  and the fact that if $1\leq p\leq 2$, then the Fourier transform is $L^{p}\rightarrow L^{q}$ bounded, where $\frac{1}{p}+\frac{1}{q}=1$, we obtain
\begin{align*}
|\widehat{u}(t)|\leq& |\widehat{u}_{0}|+\int_{0}^{t} |\xi |e^{-(t-\tau)|\xi|^{2}}
\left\{|\widehat{u\otimes u}(\tau)|+|\widehat{\nabla d\odot\nabla d}(\tau)|\right\}\text{d}\tau\nonumber\\
\leq &  |\widehat{u}_{0}|+\int_{0}^{t} |\xi| e^{-(t-\tau)|\xi|^{2}}\left(\|(u\otimes u)(\tau)\|_{L^{1}}
+\|(\nabla d\odot\nabla d)(\tau)\|_{L^{1}}\right)\text{d}\tau\nonumber\\
\leq &  |\widehat{u}_{0}|+\int_{0}^{t} |\xi| e^{-(t-\tau)|\xi|^{2}}\left(\|u(\tau)\|_{L^{2}}^{2}
+\|\nabla d(\tau)\|_{L^{2}}^{2}\right)\text{d}\tau\nonumber\\
\leq &  |\widehat{u}_{0}|+|\xi| \int_{0}^{t}  \left(\|u(\tau)\|_{L^{2}}^{2}
+\|\nabla d(\tau)\|_{L^{2}}^{2}\right)\text{d}\tau;
\end{align*}
\begin{align*}
|\widehat{\nabla d}(t)|\leq& |\widehat{\nabla d}_{0}|+\int_{0}^{t} |\xi |e^{-(t-\tau)|\xi|^{2}}
\left\{|\widehat{u\cdot\nabla d}(\tau)|+|\widehat{|\nabla d|^{2} d}(\tau)|\right\}\text{d}\tau\nonumber\\
\leq& |\widehat{\nabla d}_{0}|+\int_{0}^{t} |\xi |e^{-(t-\tau)|\xi|^{2}}
\left\{\|u\cdot\nabla d(\tau)\|_{L^{1}}+\|||\nabla d|^{2} d|(\tau)\|_{L^{1}}\right\}\text{d}\tau\nonumber\\
\leq& |\widehat{\nabla d}_{0}|+\int_{0}^{t} |\xi |e^{-(t-\tau)|\xi|^{2}}
\left\{\|u(\tau)\|_{L^{2}}\|\nabla d(\tau)\|_{L^{2}}+\|\nabla d(\tau)\|_{L^{2}}^{2} \right\}\text{d}\tau\nonumber\\
\leq&  |\widehat{\nabla d}_{0}|+ |\xi | \int_{0}^{t}
\left\{\|u(\tau)\|_{L^{2}}^{2}+\|\nabla d(\tau)\|_{L^{2}}^{2} \right\}\text{d}\tau.
\end{align*}
Combining the above two inequalities together, we conclude the proof of  \eqref{eq3.1}. This completes the proof of Lemma \ref{lem3.1}.
\end{proof}

\begin{lemma}\label{lem3.2}
For $f\in L^{p}(\mathbb{R}^{2})$, $1\leq p<2$, and let $S(t):=\{\xi\in\mathbb{R}^{2}; |\xi|\leq g(t)\}$. For a
continuous function $g:\mathbb{R}^{+}\rightarrow \mathbb{R}^{+}$. Then
\begin{align*}
\int_{S(t)} |\widehat{f}|^{2}\text{d}\xi\leq C g(t)^{2\left(\frac{2}{p}-1\right)},
\end{align*}
where the constant $C$ depends on the $L^{p}$ norm of $f$.
\end{lemma}

\begin{proof}
Notice that for $1\leq p<2$, there holds
\begin{align*}
\|\widehat{f}\|_{L^{q}}\leq C\|f\|_{L^{p}},
\end{align*}
where $\frac{1}{p}+\frac{1}{q}=1$.  Consequently, one has
\begin{align*}
\int_{S(t)} |\widehat{f}|^{2}\text{d}\xi \leq& \left\{\int_{S(t)} |\widehat{f}|^{q}\text{d}\xi\right\}^{\frac{2}{q}}
\left\{\int_{S(t)} 1\text{d}\xi\right\}^{1-\frac{2}{q}}\nonumber\\
\leq & C\|\widehat{f}\|_{L^{q}}^{2} \left\{\int_{0}^{2\pi} \int_{0}^{g(t)} r\text{d}r\text{d}\theta\right\}^{1-\frac{2}{q}}\nonumber\\
\leq& C g(t)^{ 2(1-\frac{2}{q}) }=C g(t) ^{2(\frac{2}{p}-1)}.
\end{align*}
The proof of  Lemma \ref{lem3.2}  is completed.
\end{proof}
\medskip

In order to compute the actual decay rate of the $L^{2}$ norm of the solution to system \eqref{eq1.1}--\eqref{eq1.4}  as in Theorem \ref{thm1.2}, we still need to prove the
following useful Lemma.

\begin{lemma}\label{lem3.3}
Let $(u,d)$ be a smooth solution to the system \eqref{eq1.1}--\eqref{eq1.4} with initial data
$(u_{0},d_{0})$  satisfying the initial condition as in Theorem \ref{thm1.2}. Then we have
\begin{align}\label{eq3.3}
\|u(t)\|_{L^{2}}^{2}+\|\nabla d(t)\|_{L^{2}}^{2}\leq C\ln(1+t)^{-2},
\end{align}
 where the constant $C$ depends on the $L^{p}$ and $L^{2}$ norms of the initial data.
\end{lemma}

\begin{proof}
We use the Fourier splitting method, taking
\begin{align*}
B(t):=\{\xi:|\xi|\leq g(t)\} \quad \text{ with } g(t)=\left[\frac{1}{\overline{\omega}(e+t)\ln(e+t)}\right]^{\frac{1}{2}}, \text{ and } \overline{\omega} \text{ defined as Theorem \ref{thm2.1}}.
\end{align*}
By multiplying \eqref{eq1.1} and \eqref{eq1.2} by $u$ and $-\Delta d-|\nabla d|^{2}d$ respectively, then integrating the two resulting equations with respect to $x$ over $\mathbb{R}^{2}$,  and using integration by parts, it follows that
\begin{align*}
\frac{1}{2}\frac{d}{dt} \left(\|u(t)\|_{L^{2}}^{2}+\|\nabla d(t)\|_{L^{2}}^{2}\right)
+\left(\|\nabla u\|_{L^{2}}^{2}+\|\Delta d +|\nabla d|^{2}d\|_{L^{2}}^{2}\right)=0,
\end{align*}
where we have used the fact that $|d|=1$. By using Theorem \ref{thm2.1}, it follows from the above equality that
\begin{align*}
\frac{d}{dt} \left(\|u(t)\|_{L^{2}}^{2}+\|\nabla d(t)\|_{L^{2}}^{2}\right)
\leq -2\overline{\omega}\left(\|\nabla u\|_{L^{2}}^{2}+\|\Delta d \|_{L^{2}}^{2}\right).
\end{align*}
Thus we have
\begin{align}\label{eq3.4}
\frac{d}{dt} \int_{\mathbb{R}^{2}} \left(|\widehat{u}(t)|^{2}+|\widehat{\nabla d}(t)|^{2}\right)\text{d}\xi
\leq -2\overline{\omega}\int_{\mathbb{R}^{2}}|\xi|^{2}\left(|\widehat{u}(t)|^{2}+|\widehat{\nabla d}(t)|^{2}\right)\text{d}\xi.
\end{align}
Notice that there holds
\begin{align*}
2\overline{\omega}\int_{\mathbb{R}^{2}}|\xi|^{2}\left(|\widehat{u}(t)|^{2}+|\widehat{\nabla d}(t)|^{2}\right)&\text{d}\xi  \geq 2\overline{\omega}\int_{B(t)}|\xi|^{2}\left(|\widehat{u}(t)|^{2}+|\widehat{\nabla d}(t)|^{2}\right)\text{d}\xi\nonumber \\
+&\frac{2}{(e+t)\ln(e+t)} \int_{B(t)^{c}} \left(|\widehat{u}(t)|^{2}+|\widehat{\nabla d}(t)|^{2}\right)\text{d}\xi.
\end{align*}
Hence, inequality \eqref{eq3.4} becomes
\begin{align}\label{eq3.5}
\frac{d}{dt} \int_{\mathbb{R}^{2}} \left(|\widehat{u}(t)|^{2}+|\widehat{\nabla d}(t)|^{2}\right)\text{d}\xi
+& \frac{2}{(e+t)\ln(e+t)}\int_{\mathbb{R}^{2}}|\xi|^{2}\left(|\widehat{u}(t)|^{2}+|\widehat{\nabla d}(t)|^{2}\right)\text{d}\xi\nonumber\\
\leq &\frac{2}{(e+t)\ln(e+t)}\int_{B(t)}  \left(|\widehat{u}(t)|^{2}+|\widehat{\nabla d}(t)|^{2}\right)\text{d}\xi.
\end{align}
Multiplying on both sides by $h(t)=[\ln(e+t)]^{2}$, it follows that
\begin{align}\label{eq3.6}
&\frac{d}{dt}\left\{\left(\ln(e+t)\right)^{2} \int_{\mathbb{R}^{2}} \left(|\widehat{u}(t)|^{2}+|\widehat{\nabla d}(t)|^{2}\right)\text{d}\xi\right\}\nonumber\\
\leq &\frac{2\ln(e+t)}{(e+t)}\int_{B(t)}  \left(|\widehat{u}(t)|^{2}+|\widehat{\nabla d}(t)|^{2}\right)\text{d}\xi.
\end{align}
By using Lemmas \ref{lem3.1}, \ref{lem3.2} and  energy inequality \eqref{eq2.10},  we have
\begin{align*}
&\int_{B(t)}  \left(|\widehat{u}(t)|^{2}+|\widehat{\nabla d}(t)|^{2}\right)\text{d}\xi\nonumber\\
\leq& C\left\{\int_{B(t)}(|\widehat{u}_{0}|^{2}+|\widehat{\nabla d}_{0}|^{2})\text{d}\xi
+\int_{B(t)}\int_{0}^{t}|\xi|^{2} \left(\|u(\tau)\|_{L^{2}}^{2}+\|\nabla d(\tau)\|_{L^{2}}^{2}\right)\text{d}\tau \text{d}\xi  \right\}\nonumber\\
\leq& C\left\{\int_{B(t)}(|\widehat{u}_{0}|^{2}+|\widehat{\nabla d}_{0}|^{2})\text{d}\xi
+\int_{B(t)}|\xi|^{2}t\int_{0}^{t}  \left(\|u(\tau)\|_{L^{2}}^{4}+\|\nabla d(\tau)\|_{L^{2}}^{4}\right)\text{d}\tau \text{d}\xi  \right\}
\nonumber\\
\leq& C\left\{\int_{B(t)}(|\widehat{u}_{0}|^{2}+|\widehat{\nabla d}_{0}|^{2})\text{d}\xi
+(\|u_{0}\|_{L^{2}}^{4}+\|\nabla d_{0}\|_{L^{2}}^{4})\int_{B(t)}|\xi|^{2}t^{2} \text{d}\xi  \right\}
\nonumber\\
\leq& C\left\{ \left(\frac{1}{\overline{\omega}(e+t)\ln(e+t)}\right)^{\frac{2}{p}-1} + t^{2}\int_{0}^{2\pi}
\int_{0}^{g(t)} r^{3}\text{d}r\text{d}\theta \right\}\nonumber\\
\leq & C\left\{ \left(\frac{1}{\overline{\omega}(e+t)\ln(e+t)}\right)^{\frac{2}{p}-1} + \left( \frac{t}{\overline{\omega}(e+t)\ln(e+t)}\right)^{2}\right\}.
\end{align*}
Inserting the above inequality into \eqref{eq3.6},  and integrating with  respect to $t$ over $[0,\infty)$, we get
\begin{align*}
&\left(\ln(e+t)\right)^{2} \int_{\mathbb{R}^{2}} \left(|\widehat{u}(t)|^{2}+|\widehat{\nabla d}(t)|^{2}\right)\text{d}\xi\nonumber\\
\leq& \|u_{0}\|_{L^{2}}+\|\nabla d_{0}\|_{L^{2}}^{2}+C\left\{
\int_{0}^{\infty} (e+\tau)^{-\frac{2}{p}}[\ln(e+\tau)]^{2-\frac{2}{p}}\text{d}\tau
+\int_{0}^{\infty} \frac{\tau^{2}}{(e+\tau)^{3}\ln(e+\tau)}\text{d}\tau
\right\}.
\end{align*}
Notice that the two integral in the right hand side of the above inequality is finite\footnote{For readers convenience, we show the first integral is finite. Notice that
$1\leq p<2$ implies that $1-\frac{2}{p}< 0$. Let $\tau=e^{\frac{p}{2-p}s}-e$, then we have $d\tau =\frac{p}{2-p}e^{\frac{p}{2-p}s}\text{d} s$, and the integral becomes $ \left(\frac{p}{2-p}\right)^{3}\int_{0}^{\infty} e^{-s} s^{2-\frac{2}{p}}\text{d}s :=\left(\frac{p}{2-p}\right)^{3}\Gamma (3-\frac{2}{p})<\infty$.}, hence we obtain
\begin{align*}
\int_{\mathbb{R}^{2}} \left(|\widehat{u}(t)|^{2}+|\widehat{\nabla d}(t)|^{2}\right)\text{d}\xi
\leq C[\ln(e+t)]^{-2},
\end{align*}
which implies \eqref{eq3.3}. This completes the proof of Lemma \ref{lem3.3}.
\end{proof}
\medskip

By using Lemmas \ref{lem3.1}--\ref{lem3.3},  we now present the proof of our main results.
\medskip
\\
\textbf{Proof of Theorem \ref{thm1.2}.} \  We proceed as in Lemma \ref{lem3.3}, employing the Fourier splitting method again. By defining
\begin{align*}
B(t):=\{\xi:|\xi|\leq g(t)\} \quad \text{ with } g(t)=\left[\frac{1}{2\overline{\omega}(1+t)}\right]^{\frac{1}{2}}, \text{ and } \overline{\omega} \text{ defined as Theorem \ref{thm2.1}}.
\end{align*}
Then, we have
\begin{align*}
2\overline{\omega}\int_{\mathbb{R}^{2}}|\xi|^{2}\left(|\widehat{u}(t)|^{2}+|\widehat{\nabla d}(t)|^{2}\right)&\text{d}\xi  \geq 2\overline{\omega}\int_{B(t)}|\xi|^{2}\left(|\widehat{u}(t)|^{2}+|\widehat{\nabla d}(t)|^{2}\right)\text{d}\xi\nonumber \\
+&\frac{1}{(1+t)} \int_{B(t)^{c}} \left(|\widehat{u}(t)|^{2}+|\widehat{\nabla d}(t)|^{2}\right)\text{d}\xi.
\end{align*}
Similar as deriving \eqref{eq3.5}, by inserting the above inequality into \eqref{eq3.4}, we get
\begin{align*}
\frac{d}{dt}\int_{\mathbb{R}^{2}}\left(|\widehat{u}(t)|^{2}+|\widehat{\nabla d}(t)|^{2}\right)\text{d}\xi
&+\frac{1}{(1+t)}\int_{\mathbb{R}^{2}} \left(|\widehat{u}(t)|^{2}+|\widehat{\nabla d}(t)|^{2}\right)\text{d}\xi\nonumber\\
\leq& \frac{1}{1+t} \int_{B(t)} \left(|\widehat{u}(t)|^{2}+|\widehat{\nabla d}(t)|^{2}\right)\text{d}\xi.
\end{align*}
Multiplying on both sides of the above inequality by $h(t)=(1+t)$, it follows that
\begin{align*}
\frac{d}{dt}\left\{(1+t)\int_{\mathbb{R}^{2}} \left(|\widehat{u}(t)|^{2}+|\widehat{\nabla d}(t)|^{2}\right)\text{d}\xi\right\}\leq \int_{B(t)} \left(|\widehat{u}(t)|^{2}+|\widehat{\nabla d}(t)|^{2}\right)\text{d}\xi.
\end{align*}
By integrating with respect to time variable
\begin{align}\label{eq3.7}
(1+t)\int_{\mathbb{R}^{2}} \left(|\widehat{u}(t)|^{2}+|\widehat{\nabla d}(t)|^{2}\right)\text{d}\xi\leq \|\widehat{u}_{0}\|_{L^{2}}^{2}+\|\widehat{\nabla d}_{0}\|_{L^{2}}^{2}+\int_{0}^{t}\int_{B(s)} \left(|\widehat{u}(s)|^{2}+|\widehat{\nabla d}(s)|^{2}\right)\text{d}\xi\text{d}s.
\end{align}
By using Lemma \ref{lem3.1},  Lemma \ref{lem3.2} with $g(s)=\left[\frac{1}{2\overline{\omega}(s+1)}\right]^{\frac{1}{2}}$ and Lemma \ref{lem3.3},  we have
\begin{align*}
& \int_{B(s)} \left(|\widehat{u}(s)|^{2}+|\widehat{\nabla d}(s)|^{2}\right)\text{d}\xi\nonumber\\
\leq& C\left\{ \int_{B(s)} \left(|\widehat{u}_{0}|^{2}+|\widehat{\nabla d}_{0}|^{2}\right)\text{d}\xi
+\int_{B(s)} \int_{0}^{s} |\xi|^{2}\left(\|u(\tau)\|_{L^{2}}^{2}+\|\nabla d(\tau)\|_{L^{2}}^{2}\right)\text{d}\tau\text{d}\xi\right\}\nonumber\\
\leq& C\left\{ \int_{B(s)} \left(|\widehat{u}_{0}|^{2}+|\widehat{\nabla d}_{0}|^{2}\right)\text{d}\xi
+\int_{B(s)} |\xi|^{2} s \int_{0}^{s} \left(\|u(\tau)\|_{L^{2}}^{4}+\|\nabla d(\tau)\|_{L^{2}}^{4}\right)\text{d}\tau\text{d}\xi\right\}\nonumber\\
\leq&  C\left\{ (1+s)^{-\left(\frac{2}{p}-1\right)}
+\int_{0}^{2\pi}\int_{0}^{g(s)} r^{3} s \int_{0}^{s} \left(\|u(\tau)\|_{L^{2}}^{2}+\|\nabla d(\tau)\|_{L^{2}}^{2}\right)\left[\ln(e+\tau)\right]^{-2}\text{d}\tau\text{d}r\text{d}\theta\right\}\nonumber\\
\leq & C \left\{ (1+s)^{-\left(\frac{2}{p}-1\right)}
+\int_{0}^{s} \left(\|u(\tau)\|_{L^{2}}^{2}+\|\nabla d(\tau)\|_{L^{2}}^{2}\right)\left[\ln(e+\tau)\right]^{-2}s (1+s)^{-2}\text{d}\tau\right\}.
\end{align*}
Inserting the above inequality into \eqref{eq3.7}, it follows that
\begin{align}\label{eq3.8}
&(1+t)\int_{\mathbb{R}^{2}} \left(|\widehat{u}(t)|^{2}+|\widehat{\nabla d}(t)|^{2}\right)\text{d}\xi\leq \|\widehat{u}_{0}\|_{L^{2}}^{2}+\|\widehat{\nabla d}\|_{L^{2}}^{2}\nonumber\\
+&C\left\{ (1+t)^{-\left(\frac{2}{p}-1\right)+1}
+\int_{0}^{t}\int_{0}^{s} \left(\|u(\tau)\|_{L^{2}}^{2}+\|\nabla d(\tau)\|_{L^{2}}^{2}\right)\left[\ln(e+\tau)\right]^{-2}s (1+s)^{-2}\text{d}\tau\text{d}s\right\}.
\end{align}
Notice that
\begin{align*}
&\int_{0}^{t}\int_{0}^{s} \left(\|u(\tau)\|_{L^{2}}^{2}+\|\nabla d(\tau)\|_{L^{2}}^{2}\right)\left[\ln(e+\tau)\right]^{-2}s (1+s)^{-2}\text{d}\tau\text{d}s\nonumber\\
\leq& \int_{0}^{t} (1+s)^{-1}\text{d}s \int_{0}^{t}  \left(\|u(\tau)\|_{L^{2}}^{2}+\|\nabla d(\tau)\|_{L^{2}}^{2}\right)\left[\ln(e+\tau)\right]^{-2}\text{d}\tau\nonumber\\
\leq & C \int_{0}^{t}  (1+\tau)\left(\|u(\tau)\|_{L^{2}}^{2}+\|\nabla d(\tau)\|_{L^{2}}^{2}\right)\frac{1}{\left[\ln(e+\tau)\right]^{2}(1+\tau)}\text{d}\tau.
\end{align*}
Hence, by taking
\begin{align*}
f(t):=(1+t)\left(\|u(t)\|_{L^{2}}^{2}+\|\nabla d(t)\|_{L^{2}}^{2}\right), \quad a(t): = (1+t)^{-\left(\frac{2}{p}-1\right)+1} \text{ and } b(t) :=\frac{1}{\left[\ln(e+\tau)\right]^{2}(1+\tau)},
\end{align*}
then inequality  \eqref{eq3.8} becomes
\begin{align*}
f(t)\leq \|\widehat{u}_{0}\|_{L^{2}}^{2}+\|\widehat{\nabla d}_{0}\|_{L^{2}}^{2}+Ca(t) +C\int_{0}^{t} f(\tau)b(\tau)\text{d}\tau,
\end{align*}
which together with Gronwall's inequality implies that
\begin{align*}
f(t)\leq f(0)\exp\left(C\int_{0}^{t}b(\tau)\text{d}\tau\right)+\int_{0}^{t} Ca'(\tau)\exp\left(C\int_{\tau}^{t}b(s)\text{d}s\right)\text{d}\tau.
\end{align*}
Notice that it is easy to see
\begin{align*}
\int_{0}^{t}b(\tau)\text{d}\tau<\infty\quad \text{for all }t>0.
\end{align*}
Hence, inequality  \eqref{eq3.8} becomes
\begin{align*}
(1+t)\left(\|u(t)\|_{L^{2}}^{2}+\|\nabla d(t)\|_{L^{2}}^{2}\right)\leq& \|\widehat{u}_{0}\|_{L^{2}}^{2}+\|\widehat{\nabla d}_{0}\|_{L^{2}}^{2}+C\int_{0}^{t} (1+\tau)^{-\left(\frac{2}{p}-1\right)}\text{d}\tau\nonumber\\
\leq & \|\widehat{u}_{0}\|_{L^{2}}^{2}+\|\widehat{\nabla d}_{0}\|_{L^{2}}^{2}+C(1+t)^{-\left(\frac{2}{p}-1\right)+1},
\end{align*}
which implies that
\begin{align*}
\|u(t)\|_{L^{2}}^{2}+\|\nabla d(t)\|_{L^{2}}^{2}\leq C(1+t)^{-1}+ C(1+t)^{-\left(\frac{2}{p}-1\right)}\leq C(1+t)^{-\left(\frac{2}{p}-1\right)}.
\end{align*}
Thus we complete the proof of Theorem \ref{thm1.2}.
\hfill$\Box$
\\
\\

\end{document}